\documentclass[11pt]{amsart}
\usepackage{amsmath,amsthm,amsfonts,amssymb,latexsym}
\usepackage[all,2cell,dvips]{xy}

 \usepackage{amsaddr}

\newtheorem{theorem}{Theorem}[section]
\newtheorem{lemma}[theorem]{Lemma}
\newtheorem{proposition}[theorem]{Proposition}
\newtheorem{corollary}[theorem]{Corollary}

\theoremstyle{definition}
\newtheorem{definition}[theorem]{Definition}

\theoremstyle{remark}

\newtheorem{fact}{Fact}

\newcommand{\Q}{\mathbb{Q}}

\newcommand{\Mr}{\operatorname{MR}}

\newcommand{\tp}{\operatorname{tp}}

\newcommand{\trdeg}{\operatorname{trdeg}}

\newcommand{\rk}{\operatorname{rk}}

\newcommand{\scl}{\operatorname{scl}}
\newcommand{\sdim}{\operatorname{sdim}}

\newcommand{\bk}{\mathbf{k}}

\newcommand{\Cal}{\mathcal}

\def \<{\langle}
\def \>{\rangle}

\def \((  {(\!(}
\def \)) {)\!)}

\numberwithin{equation}{section}

\allowdisplaybreaks[2]

\begin{document}

\bibliographystyle{plain}
\title{Topological study of pairs of algebraically closed fields}
 
\author{Ayhan G\"{u}nayd\i n}
\address{Bo\u{g}azi\c{c}i \"{U}niversitesi, Istanbul, Turkey}
\email{ayhan.gunaydin@boun.edu.tr}

\thanks{This work was partially supported by T\"{U}B\.{I}TAK Career Grant 113F119. }

\date{\today}

\maketitle
\begin{abstract}
We construct a topology on a given algebraically closed field with a distinguished subfield which is also algebraically closed. This topology is finer than Zariski topology and it captures the sets definable in the pair of algebraically closed fields as above; in the sense that definable sets are exactly the constructible sets in this topology.
\end{abstract}

\section{Introduction}
The study of pairs of algebraically closed fields goes back to Keisler; he proves in \cite{Keisler} that the theory of proper pairs of algebraically closed fields of given characteristic is complete. One can extract a quantifier elimination result from his proof which we restate in Section \ref{definable_sets_section}.

\medskip\noindent
We study pairs of algebraically closed fields in  the language $\{+,\cdot,0,1,U\}$, where $U$ is a unary predicate, which is interpreted as the smaller field and we write them as $(\Omega,\bk)$.

\medskip\noindent
It is worth mentioning that recently, Delon proved a quantifier elimination in a richer language; see \cite{Delon}.

\medskip\noindent
In this paper, we suggest a topological study of these pairs in characteristic zero. Each such pair, $(\Omega,\bk)$ has an elementary extension $(\Omega^*,\bk^*)$ such that $\Omega^*$ has a derivation on it, with which it becomes a differentially closed field and $\bk^*$ is the constant field for this derivation.  For our purposes, we may work in $(\Omega^*,\bk^*)$, which we do and still call it $(\Omega,\bk)$. For each $n>0$, we have the Kolchin topology on $\Omega^n$. However, there are Kolchin closed sets that are not definable in the pair. Therefore, we have to choose some of them to construe a useful topology. We could just declare that the closed sets are the Kolchin closed sets that are definable in the pair, but this would not be very useful in understanding the definable sets.  We introduce a topology in Section \ref{topology_sec} in detail. Here we just say that it is the coarsest topology in which the set of $n$-tuples that are linearly dependent over $\bk$ is closed and the polynomial maps are continuous. We call this topology the {\it pair topology} and we prove the following. 

\begin{theorem}\label{MainTh}
Every set definable in the pair $(\Omega,\bk)$ is a boolean combination of pair-closed sets.
\end{theorem}


\medskip\noindent
In an attempt to prove that the pair-closed sets are exactly the Kolchin closed sets that are definable, we could not overcome the following version of Kolchin Irreducibility Theorem in our setting.

\medskip\noindent
{\bf Question.} Is it true that each pair-irreducible pair-closed set is also Kolchin-irreducible?


\medskip\noindent
We prove the following in Section \ref{definable_sets_section}.

\begin{proposition}\label{main_prop}
If the answer to the question above is affirmative, then each Kolchin-closed definable set is pair-closed.
\end{proposition}

\medskip\noindent
The complete theory of proper pairs of algebraically closed fields of characteristic zero is $\omega$-stable of Morley rank $\omega$. We investigate the relation of Morley rank with pair topology and among others we prove the following.

\begin{proposition}
Let $X$ be definable in $(\Omega,\bk)$. Then $\Mr(\overline X)=\Mr(X)$.
\end{proposition}

\medskip\noindent
We also have another notion of dimension coming from a pregeometry called {\it small closure}. Small sets are the images of $\bk^n$ under multivalued functions that are definable in the field $\Omega$ and small closure is the similar to the usual model theoretic algebraic closure: we replace {\it finite} by {\it small}. We explain this in detail in Section \ref{ranks_section}. 

\medskip\noindent
Small dimension refines Morley rank in some way. We prove the version of the proposition above for small dimension.

\section{The topology}\label{topology_sec}

\noindent
Let $\Omega$ be an algebraically closed field of characteristic zero and let $\bk$ be a proper subfield which is also algebraically closed. We study the pair $(\Omega, \bk)$ in the language $L(U)=\{+,\cdot,0,1,U\}$ extending the language $L=\{+,\cdot,0,1\}$ of rings by a unary predicate $U$. We set the word {\it definable} to mean definable in $L(U)$ with parameters; otherwise we specify the language.

\medskip\noindent
As mentioned in the introduction, we may assume that $\Omega$ is equipped with a derivation with which it becomes a differentially closed field and $\bk$ is the constant field of that derivation. 
Then for each $n>0$, $\Omega^n$ has the Kolchin topology on it. Recall that these topologies are noetherian. As usual, most of the times we do not specify $n$ when talking about Kolchin topology and for the ease of nation, we call the Kolchin closed sets as K-closed sets.

\medskip\noindent
Of course, not all K-closed sets are definable in $(\Omega,\bk)$. For instance, the graph of derivation is not definable and a non-definable K-closed subset of $\Omega$ is given by 	
\[
\delta(X)=X^3-X^2.
\]
(We explain why this is not definable in Section \ref{ranks_section}.)
With this in mind, the first attempt to define a topology would be to define the closed sets to be K-closed sets that are definable in $(\Omega,\bk)$. This is a topology, however it is not easy to work with it unless we have more information on how the closed sets look like. 

\medskip\noindent
Before introducing the pair topology, we would like to mention the definable functions appearing in Delon's quantifier elimination result in \cite{Delon}. For $n>0$ and $i\in\{1,\dots,n\}$ the function $f_{n,i}$ is defined as follows:

\begin{align*}
f_{n,i}(\alpha_1,\dots,\alpha_n,\beta)=\gamma \Leftrightarrow& \quad \alpha_1,\dots,\alpha_n\text{ are linearly independent over }\bk,\\
& \quad \beta=a_1\alpha_1+\dots+a_n\alpha_n\text{ for some }a_1,\dots,a_n\in\bk\\
& \quad \text{ and }\gamma=a_i.
\end{align*}

\medskip\noindent
Note that the functions $f_{n,i}$ are partial functions. We might send everything not in the domain to a fixed $t\in\Omega\setminus \bk$, however we do not elaborate on this detail here. For $n>0$, each $f_{n,i}$ have the same domain and we denote it by $X_{n}$.

%

\medskip\noindent
Recall that for $n>0$, the Wronskian is defined as the determinant
\[
W_n(x_1,\dots,x_n)=
\det
 \begin{pmatrix}
  x_1 & x_2 &\cdots & x_n\\
  x_1' & x_2' &\cdots & x_n'\\
  \vdots & \vdots & \ddots & \vdots\\
  x_1^{(n-1)} & x_2^{(n-1)} &\cdots & x_n^{(n-1)}
 \end{pmatrix}.
\]
This is a differential polynomial in indeterminates $x_1,\dots,x_n$ of degree $n$ and order $n-1$. Hence its zero set is a K-closed subset of $\Omega^n$. Note also that its zero set is definable in the pair $(\Omega,\bk)$ since
\[
W_n(\alpha_1,\dots,\alpha_n)=0\Leftrightarrow \alpha_1,\dots,\alpha_n \text{ are linearly dependent over }\bk.
\]
So for our purposes this set should be closed. Let $Y_n$ denote this set.

\medskip\noindent
Other than this, we also want polynomial functions to be continuous. Therefore, we would like the sets $f^{-1}(Y_k)\subseteq\Omega^n$, where $f:\Omega^n\to\Omega^k$ is a polynomial function to be closed. The collection of such sets are not closed under intersection, so we define a {\it basic closed set} to be 
\[
f^{-1}(Y_{k_1}\times\dots\times Y_{k_m}),
\]
where $f:\Omega^n\to\Omega^{k_1+\dots+k_m}$ is again a polynomial function.

\medskip\noindent
These sets are K-closed. Since Kolchin topology is noetherian, finite union of such sets form the closed sets of a topology on $\Omega^n$. This topology is coarser than the Kolchin topology. In particular, it is noetherian, too. Noting that $Y_1=\{0\}$, we conclude also that it refines the Zariski topology. We call this topology the \emph{pair topology} and we write closed, open, irreducible, etc. for topological concepts in pair topology.

\medskip\noindent
Every polynomial map is continuous. In particular, the projection maps are continuous and as a result, the product of pair topologies on $\Omega^m$ and $\Omega^n$ is coarser than the pair topology on $\Omega^{m+n}$.

\medskip\noindent
The domain $X_{n}$ of $f_{n,i}$ is constructible as it is $\big((\Omega^n\setminus Y_n)\times \Omega\big)\cap Y_{n+1}$.

\medskip\noindent
We are aiming to prove that every set definable in the pair is constructible. We first prove that $\bk^n$ is closed.

\begin{lemma}
For each $n>0$, $\bk^n$ is closed.
\end{lemma}

\begin{proof}
Let $f:\Omega\to\Omega^2$ be defined as $f(\alpha)=(\alpha,1)$. Now it is easy to check that $\bk=f^{-1}(Y_2)$.  Since the product of two closed sets is again closed, we get that $\bk^n$ is closed.
\end{proof}

\medskip\noindent
Next lemma gives some more closed sets.

\begin{lemma}
For every $\alpha_1,\dots,\alpha_n\in\Omega$, the set $\bk\alpha_1+\dots+\bk\alpha_n$ is closed.
\end{lemma}

\begin{proof}
Just note that if $\alpha_1,\dots,\alpha_n$ are linearly independent over $\bk$, then
\[
\bk\alpha_1+\dots+\bk\alpha_n=f^{-1}(Y_{n+1})
\]
for $f(\gamma)=(\alpha_1,\dots,\alpha_n,\gamma)$. 

\medskip\noindent
If $\alpha_1,\dots,\alpha_n$ are not linearly independent, then we may choose a maximal linearly independent set among them. Without loss of generality, let $\alpha_1,\dots,\alpha_m$ be such. Then $\bk\alpha_1+\dots+\bk\alpha_m$ is closed by the previous part and so is $\bk\alpha_1+\dots+\bk\alpha_n$.
\end{proof}

\medskip\noindent
Next we show that the graphs of the functions $f_{n,i}$ are constructible.

\begin{lemma}
For every $n>0$ and $0<i\leq n$, the graph of $f_{n,i}$ is constructible
\end{lemma}

\begin{proof}
Without loss of generality we take $i=1$ and consider the graph of $f_{n,1}$:
\[
\Gamma=\{(\alpha_1,\dots,\alpha_n,\beta,a)\in X_n\times\bk: \beta=a\alpha_1+\sum_{i=2}^n a_i\alpha_i \text{ for some }a_2,\dots,a_n\in\bk\}.
\]
Define $f:\Omega^{n+2}\to\Omega^n$ by 
\[
f(x_1,\dots,x_n,y,z)=(z x_1-y,x_2,\dots,x_n)
\]
Now it is clear that 
\[
\Gamma=\big((\Omega^n\setminus Y_n)\times\Omega\times\bk\big)\cap f^{-1}(Y_n).
\]
It is now clear that the set on the right is constructible and hence the graph of $f_{n,1}$ is constructible.
\end{proof}

\medskip\noindent
It is not clear from this proof whether $f_{n,i}$ is continuous. We show that later in Proposition \ref{f's_continuous}.  We first observe that closed subsets of $\bk^n$ are just Zariski closed subsets. 
%

\begin{proposition}\label{closed_subsets_of_k}
Let $X\subset\bk^n$ be closed. Then X is a Zariski closed subset of $\bk^n$.
\end{proposition}

\begin{proof}
It suffices to prove that a set of the form $\bk^n\cap f^{-1}(Y_m)$ is Zariski closed where $f:\Omega^n\to\Omega^m$ is a polynomial map, because Zariski closed sets are closed under finite unions and intersections.

\medskip\noindent
Let $f=(f_1,\dots,f_m)$. We need to show that the set 
\[
X=\{a\in \bk^n: \sum_{i=1}^m b_i f_i(a)=0 \text{ for some }b_1,\dots,b_m\in\bk\text{ that are not all $0$}\}
\]
is Zariski closed. Let $A=\{t_1,t_2,\dots,t_p\}\subseteq\Omega$ be a $\bk$-linearly independent set such that for each $i$, the coefficients of $f_i$ are from $\bk t_1+\dots+\bk t_p$.

\medskip\noindent
After rearranging we have
\[
X=\left\{a\in\bk^n:g_1(a,b)t_1+\dots+g_p(a,b)t_p=0\text{ for some }b\in\bk^m\setminus\{\vec 0\}\right\}
\]
where $g_1(x,y),\dots,g_p(x,y)$ are polynomials in $\bk[x,y]$ whose degrees in the variable $y_j$ are at most $1$ for every $j$.

\medskip\noindent
Then
\[
X=\left\{a\in\bk^n:g_1(a,b)=\dots=g_p(a,b)=0\text{ for some }b\in\bk^m\setminus\{\vec 0\}\right\}
\]
So $a$ being in this set is equivalent to a certain system of linear equations over $\bk$ (depending on $a$) having a nonzero solution. This last condition is equivalent to certain polynomials over $\bk$ having a common zero at $a$. Therefore $X$ is Zariski closed.
\end{proof}

\begin{corollary}
Proper closed subsets of $\bk$ are finite sets.
\end{corollary}

\medskip\noindent
Now we are ready to prove that the functions $f_{n,i}$ are continuous.

\begin{proposition}\label{f's_continuous}
For $n>0$ and $i\in\{1,\dots,n\}$, the function, $f_{n,i}:X_n\to\bk$ is continuous.
\end{proposition}

\begin{proof}
We take $i=1$. By the previous corollary, it suffices to prove that the preimages of  singletons are closed in $X_n$. So let $a\in\bk$ and consider $f_{n,1}^{-1}(a)$. This preimage  is 
\[
\{(\alpha_1,\dots,\alpha_n,\beta)\in X_n:a\alpha_1+a_2\alpha_2+\dots+a_n \alpha_n=\beta\text{ for some }a_2,\dots,a_n\in\bk\}.
\]
It now easy to see that this is nothing other than $f^{-1}(Y_n)\cap X_n$ where 
\[f(\alpha_1,\dots,\alpha_n,\beta)=(a\alpha_1-\beta,\alpha_2,\dots,\alpha_n).\]
 So $f_{n,1}^{-1}(a)$ is closed.
\end{proof}

\section{Definable Sets}\label{definable_sets_section}
\noindent
Our aim in this section is to prove that sets definable in the pair $(\Omega,\bk)$ are exactly the constructible sets. 

\medskip\noindent
We first make a reduction to parameter-free sets. 

\begin{lemma}\label{red_parameter_free}
Suppose that for each $n>0$, every subset of $\Omega^n$ that is $\emptyset$-definable is constructible. Then every definable set is constructible. Moreover, if $Y$ is $A$-definable, then $Y$ is a boolean combination of closed sets definable over $A$, provided that $\emptyset$-definable sets are boolean combinations of $\emptyset$-definable closed sets.
\end{lemma}

\begin{proof}
Let $Y\subseteq \Omega^m$ be an arbitrary definable set and let $X\subseteq\Omega^{m+n}$ be a parameter-free definable set and $a\in\Omega^n$ such that $Y=X(a)$.

\medskip\noindent
By assumption $X$ is constructible. 

\medskip\noindent
Consider the polynomial map from $\Omega^m$ to $\Omega^{m+n}$ defined by $f(y)=(y,a)$. It is clear that $Y=f^{-1}(X\cap(\Omega^m\times\{a\}))$. So being the preimage of a constructible set under a continuous map, $Y$ is constructible.

\medskip\noindent
Note that closed and open sets appearing in the proof above are definable over the same parameters as $Y$. Therefore the last part of the proposition follows.
\end{proof}

\medskip\noindent
Before going through some more reductions, we present a quantifier elimination for pairs of algebraically closed fields. This can be obtained from the proof of completeness in \cite{Keisler}, but it also follows from Theorem 3.8 of  \cite{vdDG}.

\begin{fact}\label{near_MC}
Every $\emptyset$-definable subset of $\Omega^n$ is a boolean combination of sets defined by formulas of the form
\[
\exists y_1\cdots\exists y_m\big(\bigwedge_{r=1}^m U(y_r)\wedge\phi(x,y) \big),
\]
where $x$ is an $n$-tuple of variables and $\phi(x,y)$ is a quantifier-free formula in the language of rings. 
\end{fact}

\medskip\noindent
Quantifier-free formulas in the language of rings in variables $(x,y)$ are equivalent to formulas of the form
\[
\bigvee_{i=1}^k(\bigwedge_{j=1}^{s_i}p_{ij}(x,y)=0\wedge \bigwedge_{j=1}^{t_i}q_{ij}(x,y)\neq 0),
\]
where $p_{ij}$'s and $q_{ij}$'s are polynomials over $\Q$.
However, in the theory of (algebraically closed) fields, the last part is equivalent to
\[
 \prod_{j=1}^{t_i}q_{ij}(x,y)\neq 0
\]
Therefore, using Lemma \ref{red_parameter_free} and Fact \ref{near_MC}, in order to show that every $\emptyset$-definable set is constructible, we need to show that sets defined by formulas of the form
\[
\exists y_1\cdots\exists y_m\left(\bigwedge_{r=1}^m U(r_i)\wedge\bigvee_{i=1}^k  \big(q_{i}(x,y)\neq 0\wedge \bigwedge_{j=1}^{s_i}p_{ij}(x,y)=0\big)\right)
\]
are constructible.

\medskip\noindent
Finally, this last formula is equivalent to 
\[
\bigvee_{i=1}^k\exists y_1^{(i)}\cdots\exists y_m^{(i)}\left(\bigwedge_{r=1}^m U(y_r^{(i)})\wedge \big(q_{i}(x,y^{(i)})\neq 0\wedge \bigwedge_{j=1}^{s_i}p_{ij}(x,y^{(i)})=0\big)\right)
\]
So we only consider sets defined by formulas
\begin{equation}\label{basic-formula}
\exists y_1\cdots\exists y_m\big(\bigwedge_{r=1}^m U(y_r)\wedge p_0(x,y)\neq 0\wedge \bigwedge_{j=1}^{s}p_{j}(x,y)=0\big),
\end{equation}
where $x=(x_1,\dots,x_n)$ and $y=(y_1,\dots,y_m)$ are tuples of variable and $p_j$ is a polynomial over $\Q$ for each $j$.

\medskip\noindent
Now the following finishes the proof of Theorem \ref{MainTh} in a stronger way.

\begin{theorem}\label{MainThv2}
Let $X\subseteq \Omega^n$ be defined by a formula of the form (\ref{basic-formula}).
Then $X$ is a boolean combination of $\emptyset$-definable closed sets.
\end{theorem}

\begin{proof}
Let $X$ be defined by 
\[
\exists y_1\cdots\exists y_m\big(\bigwedge_{r=1}^m U(y_r)\wedge p_0(x,y)\neq 0\wedge \bigwedge_{j=1}^{s}p_{j}(x,y)=0\big).
\]

\medskip\noindent
For $j=0,1,\dots,s$, write 
\[
p_j(x,y)=\sum_{\iota\in I_j} p_{j\iota}(y) x^\iota,
\]
where $I_j$ is a finite set of multi-indices and $p_{j\iota}\in\Q[y]$ for each $\iota\in I_j$. 

\medskip\noindent
Let $K_j\subseteq I_j$ for each $j$ and let  $K=K_0\times K_1\times\dots\times K_s$. We define $S_K$ to be the set of $\alpha\in\Omega^n$ such that for each $j$ the set $\{\alpha^\iota:\iota\in K_j\}$ is linearly independent over $\bk$ and $\alpha^{\iota'}$ is in the $\bk$-linear space generated by $\{\alpha^\iota:\iota\in K_j\}$ for $\iota'\in I_j\setminus K_j$.

\medskip\noindent
Note that 
\[
X=\bigcup_K X\cap S_K
\]
where $K$ runs through the subsets of $I_0\times I_1\times\dots\times I_s$ of the form $K=K_0\times K_1\times\dots\times K_s$
Therefore it suffices to show that given such $K$ the set 
\[
X_K:=X\cap S_{K}
\]
is constructible.

\medskip\noindent
Let $k_j=|K_j|$ and $l_j=|I_j|$, and enumerate $K_j$ as 
\[
K_j=\{\iota_{j1},\dots,\iota_{jk_j}\}
\]
Let $\alpha\in S_{K}$. Then for $j=0,\dots,s$ and $\iota\in I_j$ write
\[
\alpha^{\iota}=\sum_{k=1}^{k_j} f_{k_j,k}(\alpha^{\iota_{j1}},\dots,\alpha^{\iota_{jk_j}},\alpha^{\iota})\alpha^{\iota_{jk}}
\]
As a result, for $\alpha\in S_K$ we get that
\[
p_j(\alpha,y)=\sum_{k=1}^{k_j} \sum_{\iota\in I_j} p_{j\iota}(y)f_{k_j,k}(\alpha^{\iota_{j1}},\dots,\alpha^{\iota_{jk_j}},\alpha^{\iota})\alpha^{\iota_{jk}}
\]

\medskip\noindent
So elements of $X_K$ are exactly $\alpha\in\Omega^n$ such that there is $a\in\bk^m$ with the property that for every $j=1,\dots,s$ and $k=1,\dots,l_j$
\[
\sum_{\iota\in I_j}p_{j\iota}(a)f_{l_j,k}(\alpha^{\iota_{j1}},\dots,\alpha^{\iota_{jk_j}},\alpha^{\iota})=0
\]
and 
\[
\sum_{\iota\in I_0}p_{0\iota}(a)f_{l_0,k'}(\alpha^{\iota_{01}},\dots,\alpha^{\iota_{0k_0}},\alpha^{\iota})\neq 0
\]
for some $k'\in\{1,\dots,k_0\}$.

\medskip\noindent
Now consider the definable set $Z$ containing $(A_{j,k,\iota})\in\bk^{k_0l_0+k_1l_1+\dots+k_sl_s}$ such that there is $a\in\bk^m$ with the property that
\[
\sum_{\iota\in I_0}p_{0\iota}(a)A_{0,k',\iota}\neq 0
\]
for some $k'\in\{1,\dots,k_0\}$ and 
\[
\sum_{\iota\in I_j}p_{j\iota}(a)A_{j,k,\iota}=0
\]
for each $j=1,\dots,s$ and $k=1,\dots,k_j$.

\medskip\noindent
Being a definable subset of $\bk^{k_0l_0+k_1l_1+\dots+k_sl_s}$, $Z$ is constructible by Proposition \ref{closed_subsets_of_k}.  

\medskip\noindent
Now define $f:S_K\to \bk^{k+0l_0+k_1l_1+\dots+k_sl_s}$ as follows
\[
f(\alpha)=(f_{k_j,k}(\alpha^{\iota_{j1}},\dots,\alpha^{\iota_{jk_j}},\alpha^{\iota}):j=0,\dots,s,k=1,\dots,k_j, \iota\in I_j)
\]
Now $f$ is continuous by Proposition \ref{f's_continuous}. So $X_K$ is constructible as it is $f^{-1}(Z)$.
\end{proof}

\noindent
Taking the parameters into consideration we have the following consequences.

\begin{corollary}\label{closed_sets_determine_types}
Let $(\Omega,\bk)$ be $\kappa$-saturated and let $A\subseteq \Omega$ be of cardinality less than $\kappa$ and $a\in\Omega^n$. Then $\tp(a/A)$ is determined by the closed sets in it.
\end{corollary}

\begin{corollary}\label{closure_parameters}
Let $X$ be $A$-definable. Then $\overline X$ is also $A$-definable.
\end{corollary}

\medskip\noindent
Combining Theorem \ref{MainThv2} and Proposition \ref{closed_subsets_of_k}, we get the following stable embeddedness result which is folklore.

\begin{corollary}\label{subsets_of_k}
A subset of $\bk^n$ is definable in the pair if and only if it is definable in the field $\bk$.
\end{corollary}

%

\subsection{K-Closed versus Closed}\label{K-irreducible_section}

\medskip\noindent
The well-known Kolchin Irreducibility Theorem states that a Zariski closed set in a differentially closed field is Kolchin irreducible if it is Zariski irreducible. As we have mentioned in the introduction, we need the answer to the following version of this theorem to be affirmative for the next proposition. 

\medskip\noindent
{\bf Question.} Is it correct that any irreducible closed set K-irreducible?

%
%

\medskip\noindent
Following is Proposition \ref{main_prop} from the introduction.

\begin{proposition}
Suppose that the answer to the question above is affirmative and let $X\subseteq \Omega^n$ be definable in $(\Omega,\bk)$. Then $X$ is closed if and only if it is K-closed.
\end{proposition}

\begin{proof}
It is clear that closed sets are K-closed, so we only need to prove the other implication. So let's suppose that $X$ is K-closed and prove that it is closed. 

\medskip\noindent
By Theorem \ref{MainThv2}, we may write 
\[
X=(C_1\cap U_1)\cup\dots\cup(C_m\cap U_m)
\]
where the sets $C_i$ are closed and the sets $U_i$ are open. We may also assume that each $C_i$ is irreducible. 
Then they are indeed K-irreducible by the assumption. So 
\[
X=\overline X^K=\overline{C_1\cap U_1}^K\cup\dots\cup\overline{C_m\cap U_m}^K=C_1\cup\dots\cup C_m.
\]
Thus $X$ is closed.

\end{proof}

\section{Ranks}\label{ranks_section}

\noindent
It is well-known that the theory of pairs of algebraically closed fields is $\omega$-stable with $\Mr(\bk)=1$ and $\Mr(\Omega)=\omega$.  (Here and below, Morley ranks are with respect to the pair $(\Omega,\bk)$.)

\medskip\noindent
We relate the Morley rank and the pair topology. We also relate them to another notion of dimension given by a certain pregeometry called {\it small closure}. The concept of smallness below is defined for any structure and for any subset of the underlying set, but here we always use it for subsets of an algebraically closed field and mostly with subsets definable in a pair of algebraically closed fields.

\begin{definition}
Let $\Cal M=(M,\dots)$ be a first order structure in a language $\Cal L$. We say that a subset $X$ of $M$ is \emph{large} if there is a multi-valued function $f: M^m \stackrel{n}{\longrightarrow} M$ definable in $\Cal M$ such that $f(X^m)=M$. (A {\it multi-valued function} $f:X\stackrel{n}{\longrightarrow} Y$ is a function $f:X\to \mathcal P(Y)$ such that  $f(x)$ has at most $n$ elements for each $x\in X$.)

\medskip\noindent
If $X$ is not large, then we say it is \emph{small}.
\end{definition}

\medskip\noindent
Some remarks are in order. The use of multi-valued functions in this definition is really crucial in the general setting, however it is proven in \cite{vdDG} that in the case of algebraically closed fields usual functions are enough. Note also that the notion of smallness is first order in the language $\Cal L(U)$ extending $\Cal L$ by a (new) unary predicate $U$ which is interpreted as $X$.  For more facts on the notion, the reader could check \cite{vdDG}.

\medskip\noindent
If $X$ is a subfield of an algebraically closed field $\Omega$, then the only way it can be large is that either it is the whole $\Omega$ or it is a real closed subfield such that the degree of $\Omega$ over $X$ is 2. Hence a proper algebraically closed subfield $\bk$ is always small. As a result, the image of $\bk^n$ under a definable multi-valued function is also small. The converse is also correct: a small set is contained in the image of $\bk^n$ under a definable function. Using this, it is also easy to see that being small for a set definable in $(\Omega,\bk)$ is the same as having finite Morley rank. So we freely use each one of these equivalent concepts.

\medskip\noindent
We define \emph{small closure} in a similar way to algebraic closure, replacing ``finite" by ``small". More precisely: Let $A\subseteq \Omega$, then $\alpha\in\Omega$ is in the small closure of $A$, denoted by $\operatorname{scl}(A)$, if it is contained in a small set definable in $(\Omega,\bk)$ over $A$. In a saturated enough extension, this is the same as $\overline{\bk(A)}$. Then this closure operator is a pregeometry and hence gives the notions of independence and dimension in the usual way. 

\begin{definition}
 \begin{enumerate}
  \item Given $\alpha=(\alpha_1,\dots,\alpha_n)\in\Omega^n$ and $A\subseteq \Omega$, we define the \emph{small rank}, $\operatorname{rk}(\alpha/A)$, of $\alpha$ over $A$ to be the pregeometry dimension of $\{\alpha_1,\dots,\alpha_n\}$ over $A$.
  \item The \emph{small dimension}, $\operatorname{sdim}(X)$, of a definable set $X\subseteq\Omega^n$ is defined to be the the maximum of the set
  \[
  \left\{\operatorname{rk}(\alpha/A):\alpha\in X, A\text{ is any set over which $X$ is defined}\right\}.
  \]
 \end{enumerate}
\end{definition}

\medskip\noindent
The small dimension of a definable set is well-defined, because of the explanation on page 315 of \cite{Gagelman}. 

\medskip\noindent
For instance, the small dimension of a small set is $0$ and $\sdim(\Omega^n)=n$. Note that for two definable sets $X,Y$, if $\operatorname{sdim}(X)< \operatorname{sdim}(Y)$, then $\Mr(X)<\Mr(Y)$. However, the converse is not correct: 
\[
\operatorname{sdim}(\bk)=\operatorname{sdim}(\bk+\bk \alpha)=0,
\]
 but if $\alpha\notin\bk$, then $\Mr(\bk+\bk \alpha)=2$.
 
 \medskip\noindent
 Consider the example from the introduction:
 \[
 C=\{\alpha\in\Omega: \delta(\alpha)=\alpha^3-\alpha^2\}.
 \]
 This is a strongly minimal set in the differentially closed field $(\Omega,\delta)$ and its geometry is trivial (see \cite{Marker_MT_diff_fields} for details).  If it were definable in the pair, then it would still have Morley rank at most $1$, hence it would be a small set. Then its small dimension would be $0$. However, this contradicts with the fact that its geometry is trivial. Therefore this is a subset of $\Omega$ definable in the differential field, but not in the pair.

\begin{lemma}
If $C\subsetneq\Omega^n$ is closed then $\operatorname{sdim}(C)<n$.
\end{lemma}

\begin{proof}
First let $C=Y_n$ and let $\alpha=(\alpha_1,\dots,\alpha_n)\in Y_n$. Then --without loss of generality-- we have that $\alpha_1\in \bk \alpha_2+\dots+\bk \alpha_n$. So $\operatorname{rk}(\alpha/\emptyset)$ is at most $n-1$. Therefore $\operatorname{sdim}(C)\leq n-1$.

\medskip\noindent
Now if $X$ is the pre-image of $Y_k$ under a non-zero polynomial map, then we again get a similar dependence over $\bk$. 

\medskip\noindent
It is clear that the small dimension of finite intersection is at most the least of small dimensions of sets we intersect. Similarly the small dimension of a finite union is the maximal of the small dimensions of the sets we put together. 
\end{proof}

\noindent
We collect some consequences of this lemma.

\begin{corollary}
Proper closed subsets of $\Omega$ are small.
\end{corollary}

\begin{corollary}
The Morley rank of a proper closed subset of $\Omega^n$ is less than $\omega\, n=\Mr(\Omega^n)$.
\end{corollary}

\begin{corollary}
The small dimension of a non-empty open subset of $\Omega^n$ is $n$ and its Morley rank is $\omega\, n$.
\end{corollary}

\begin{proof}
Let $U$ be an open subset of $\Omega^n$ and let $C$ be its complement in $\Omega^n$. Then one of them has to have small dimension equal to $n$. Since $\operatorname{sdim}(C)< n$, we get that $\operatorname{sdim}(U)=n$. 

\medskip\noindent
Similarly, one of $C$ or $U$ should have the same Morley rank as $\Omega^n$ and the Morley rank of $C$ is strictly less than that. 
\end{proof}

\medskip\noindent
Next we prove a partial inverse of the last corollary.

\begin{proposition}
Let $X\subseteq \Omega^n$ be a definable set that has the same Morley rank as $\Omega^n$. Then $X$ has non-empty interior. 
\end{proposition}

\begin{proof}
First note that $\overline {X}=\Omega^n$. Otherwise, by the corollary above, $\Omega^n$ would have disjoint two subsets of maximal Morley rank.

\medskip\noindent
Write 
\[
X=(C_1\cap U_1)\cup\dots\cup(C_m\cap U_m),
\]
where each $C_i$ is closed and each $U_i$ is open. Then
\[
\Omega^n=\overline X\subseteq C_1\cup\dots\cup C_m.
\]
Therefore one of the $C_i$'s must be $\Omega^n$ and this means that the corresponding $U_i$ is contained in $X$.
\end{proof}

\noindent
The following has a similar proof.

\begin{proposition}\label{nonempty_int_small}
Let $X\subseteq \Omega^n$ be a definable set that has the same small dimension as $\Omega^n$. Then $X$ has non-empty interior. 
\end{proposition}

\begin{lemma}
Let $X\subseteq \Omega^n$. Suppose that there is a projection $\pi:\Omega^n\to \Omega^k$ such that $\pi(X)$ has non-empty interior. Then $k\leq\operatorname{sdim}(X)$.
\end{lemma}

\begin{proof}
Let $X$ be definable over $A$.  Without loss of generality $\pi$ is the projection onto the first $k$ coordinates. By the proposition above $\sdim(\pi(X))=k$. Take a generic $a=(a_1,\dots,a_k)\in \pi(X)$; this is to say that $a_1,\dots,a_k$ are $\scl$-independent over $A$. Now take $a_{k+1},\dots,a_n$ such that $(a_1,\dots,a_n)\in X$. It is clear that $\rk(a_1,\dots,a_n)\geq k$. Therefore $\sdim(X)\geq k$.
\end{proof}

\begin{proposition}
Let $X\subseteq \Omega^n$. Then $\operatorname{sdim}(X)$ is the maximal $k$ such that there is a projection $\pi:\Omega^n\to \Omega^k$  in a way that $\pi(X)$ has non-empty interior. 
\end{proposition}

\begin{proof}
Let $X$ be definable over $A$ and $l=\sdim(X)$. Take $a\in X$ such that $\rk(a/A)=l$; say $a_1,\dots,a_l$ are $\scl$-independent over $A$. Then $\pi(X)$ has small dimension $l$, where $\pi$ is the projection onto the first $l$ coordinates.
\end{proof}

\medskip\noindent
\begin{lemma}\label{intersect_open}
Let $X,U$ be definable sets where $U$ is an open set intersecting $X$. Suppose also that $\overline X$ is irreducible. Then $\overline X=\overline{X\cap U}$.
\end{lemma}

\begin{proof}
Let $X=(C_1\cap U_1)\cup\dots\cup(C_m\cap U_m)$, where the sets $C_i$ are closed and irreducible and $U_i$ are open. Then 
\[\overline X=C_1\cup\dots\cup C_m.\]
 So $X=C_i$ for some $i$. Without loss of generality, $i=1$. 

\medskip\noindent
Note that the only way the conclusion of the lemma does not hold  is when $C_1\cap U_1\cap U= \emptyset$. Suppose that $C_1\cap U_1\cap U\neq \emptyset$.  Then 
\[
C_1\subseteq \Omega^n\setminus (U_1\cap U)=(\Omega^n\setminus U_1)\cup(\Omega^n\setminus U),
\]
and 
\[
C_1=(C_1\setminus U_1)\cup(C_1\setminus U).
\]
Then either $C_1\subseteq\Omega^n\setminus U_1$ or $C_1\subseteq\Omega^n\setminus U$. The former one is not possible, so we have $C_1\cap U=\emptyset$. But then $C_i=C_i\cap C_1\subseteq \Omega^n\setminus U$ and $C_i\cap U=\emptyset$, which contradicts the assumption that $X$ intersects $U$. Thus we have  $\overline{X\cap U}=\overline X$

\end{proof}

\noindent
Now we generalize Proposition \ref{nonempty_int_small}

\begin{proposition}
Let $X\subseteq \Omega^n$ be definable in the pair. Then $\operatorname{sdim}(X)=\operatorname{sdim}(\overline X)$.
\end{proposition}

\begin{proof}
First assume that $\sdim(\overline X)=n$. Since $\overline X=X\cup (\overline X\setminus X)$, either $X$ or $\overline X\setminus X$ has small dimension $n$, but $\overline X\setminus X$ has empty interior. So $\sdim(X)=n$.

\medskip\noindent
Let $\sdim(\overline X)=k$ and take $\pi:\Omega^n\to\Omega^k$ such that $\pi(\overline X)$ has non-empty interior. Let $Y=\pi(X)$. If $\sdim(X)<k$, then $\sdim Y<k$. But $\sdim(\overline Y)=\sdim(Y)$. However $\pi(\overline X)\subseteq \overline Y$, which is a contradiction.
\end{proof}

\medskip\noindent
Using Lemma \ref{intersect_open}, we have the following.

\begin{corollary}
Let $X,U\subseteq\Omega^n$ be definable and suppose that $U$ is open with $X\cap U\neq \emptyset$ and that $\overline X$ is irreducible. Then $\sdim(X)=\sdim(X\cap U)$.
\end{corollary}

\medskip\noindent
We are proceeding to prove the proposition above for Morley rank in the place of small dimension. We need the following fact from ???.

\begin{fact}
Let $a\in\Omega^n$ and $A\subseteq \Omega$ (of cardinality less than $\kappa$). Then 
\[
\Mr(a/A)=\omega\cdot \rk(a/A)+ \trdeg (\bk(a,A)/\bk(A)).
\]
\end{fact}


\begin{proposition}
Let $X\subseteq\Omega^n$ be definable. Then $\Mr(X)=\Mr(\overline{X})$.
\end{proposition}

\begin{proof}
Without loss of generality, we may assume that $X$ is of Morley degree $1$.
Let $X$ be definable over $A$, then by Corollary \ref{closure_parameters}, $\overline X$ is also definable over $A$.  We proceed by induction on $\Mr(X)$. 

\medskip\noindent 
Take a $\Mr$-generic $a\in X$ (over $A$). 
By Corollary \ref{closed_sets_determine_types}, there is a closed $Y$ defined over $A$ that contains $a$ and $\Mr(Y)=\Mr(a/A)=\Mr(X)$. Then we also have that $\Mr(X\cap Y)\geq\Mr(a/A)=\Mr(X)$. So indeed we have $\Mr(X\cap Y)=\Mr(X)$.

\medskip\noindent
As we assume that $X$ is of Morley degree $1$, we have $\Mr(X\setminus Y)<\Mr(X)$. Hence $\Mr(X\setminus Y)=\Mr(\overline{X\setminus Y})$ by the induction hypothesis.  Then 
$$
\overline X=\overline{X\cap Y}\cup\overline{X\setminus Y},
$$
and
\[
\Mr(\overline X)=\max\{\Mr(\overline{X\cap Y}),\Mr(\overline{X\setminus Y})\}=\Mr(\overline{X\cap Y}).
\]
However, $\overline{X\cap Y}\subseteq Y$ and $\Mr(\overline{X\cap Y})\leq\Mr(Y)$, which proves the proposition. 

\end{proof}

\medskip\noindent
Once again, the following is a consequence of Lemma \ref{intersect_open}.

\begin{corollary}
Let $X$ be a definable set such that $\overline X$ is irreducible and let $U$ be an open set that intersects $X$. Then $\Mr(X\cap U)=\Mr(X)$.
\end{corollary}


\begin{thebibliography}{1}

\bibitem{Delon}
F.~Delon.
\newblock Elimination des quantificaeurs dans les paires de corps
  algebriquement clos.
\newblock {\em Confluentes Mathematici}, 4(2), 2012.

\bibitem{vdDG}
Lou van~den Dries and Ayhan G{\"u}nayd{\i}n.
\newblock The fields of real and complex numbers with a small multiplicative
  group.
\newblock {\em Proc. London Math. Soc. (3)}, 93(1):43--81, 2006.

\bibitem{Gagelman}
Jerry Gagelman.
\newblock Stability in geometric theories.
\newblock {\em Ann. Pure Appl. Logic}, 132(2-3):313--326, 2005.

\bibitem{Keisler}
H.~Jerome Keisler.
\newblock Complete theories of algebraically closed fields with distinguished
  subfields.
\newblock {\em Michigan Math. J.}, 11:71--81, 1964.

\bibitem{Marker_MT_diff_fields}
David Marker.
\newblock Model theory of differential fields.
\newblock In {\em Model theory, algebra, and geometry}, volume~39 of {\em Math.
  Sci. Res. Inst. Publ.}, pages 53--63. Cambridge Univ. Press, Cambridge, 2000.

\end{thebibliography}
\end{document}